\documentclass[11pt]{amsart}

\input xy
\xyoption{all}
\usepackage{microtype}
\usepackage{tikz}
\usepackage{amsmath,amssymb,graphicx,bbm,color}

\usepackage[pdftex]{hyperref}

\definecolor{grey}{rgb}{0.75,0.75,0.75}
\definecolor{orange}{rgb}{1.0,0.5,0.5}
\definecolor{brown}{rgb}{0.5,0.25,0.0}
\definecolor{pink}{rgb}{1.0,0.5,0.5}
\newenvironment{example}{\begin{varexample}\em}{\em\end{varexample}}

\newcommand{\tmop}[1]{\ensuremath{\operatorname{#1}}}
\newcommand{\nin}{\not\in}

\newcommand{\tmmathbf}[1]{\ensuremath{\boldsymbol{#1}}}
\newcommand{\longrightarrowlim}{\mathop{\longrightarrow}\limits}
\newenvironment{note}{\begin{varnote}\em}{\em\end{varnote}}  
\newcommand{\tmdummy}{$\mbox{}$}

\newcommand{\OO}{\mathcal{O}}

\newcommand{\e}{\'e}
\newcommand{\ii}{\"i}
\newtheorem{definition}{Definition}[section]
\newtheorem{theorem}[definition]{Theorem}
\newtheorem{proposition}[definition]{Proposition}
\newtheorem{varnote}[definition]{Note}
\newtheorem{varexample}[definition]{Example}
\newtheorem{corollary}[definition]{Corollary}
\newtheorem{lemma}[definition]{Lemma}  

\newtheorem{varremark}[definition]{Remark}

\numberwithin{equation}{section}
\author{Kenneth Chan}

\title{Log terminal orders are numerically rational}

\begin{document}

\bibliographystyle{alpha}

\begin{abstract}
Noncommutative surfaces finite over their centres can be realised as 
orders
over surfaces. The aim of this paper is to present a noncommutative
generalisation of rational singularities, which we call numerical
rationality, for such orders. We show that numerical rationality is
independent of the choice of resolution. Our main result is that the log
terminal orders arising from the noncommutative minimal model program
\cite{CI1}, in particular, the canonical orders defined in \cite{CI2}, are
numerically rational. Both of these generalise well known facts about
rational singularities in commutative algebraic geometry.

\end{abstract}
\maketitle

\section{Introduction}

The noncommutative minimal model program (c.f. {\cite{CI1}}) resolves the
birational classification problem for noncommutative surfaces which are finite
over their centres. Such noncommutative surfaces are called orders. Many
concepts from Mori theory carry over to the noncommutative setting. Of
particular interest to us are the noncommutative counterparts of canonical and
log terminal singularities, called canonical and log terminal orders.
Canonical orders were studied extensively in {\cite{CI2}} as invariant rings
and in {\cite{ChanMcKay}} in the context a noncommutative version of the McKay
correspondence. More generally, log terminal orders can be viewed as
noncommutative quotient singularities. Absent from the noncommutative theory
is a notion of rational singularities for orders, and the main objective of
the present paper is to address this point.

Rational singularities are characterised by the property that their cohomology
does not change after passing to the resolution. On surfaces, rational
singularities were studied extensively in {\cite{lipman}}. A theorem of Artin
gives a numerical characterisation of rationality of surface singularities in
terms of the intersection theory on their resolutions (c.f.
{\cite{artinrational}}). Many naturally occurring singularities are rational,
examples include singularities of toric varieties and quotient singularities.
The log terminal singularities of the Mori program are also rational
singularities.

The distinguishing feature of our approach to studying noncommutative
singularities is the use of resolutions of singularities (c.f. {\cite{CI1}},
Corollary 3.6). It provides the necessary technology to investigate
rationality for singularities of orders. We define resolution of singularities
of orders in Section \ref{section-numericalrationality}. The most important
ingredient for resolutions are terminal orders, which are the smooth models in
the noncommutative Mori program. These are defined in {\cite{CI1}}, in terms
of discrepancies, and we review an equivalent definition in Section
\ref{setupforadjform}. Also important is the notion of a birational morphism
of orders, defined in {\cite{CI2}}. If we think of an order as a
noncommutative surface $X$, these concepts allow us to consider a resolution
of singularities as consisting of a smooth noncommutative surface $\tilde{X}$
together with a (noncommutative) birational morphism $\tilde{X}
\longrightarrow X$.

To continue the discussion, we first review the definition of an order over a
surface. Let $Z$ be a normal integral $k$-scheme of dimension $2$. An
$\OO_Z$-order $A$ is a coherent, torsion-free sheaf of $\OO_Z$-algebras which
is generically a $k ( Z )$-central simple algebra. Intuitively, we think of
$A$ as a sheaf of functions on some ``noncommutative space.'' We restrict our
attention mostly to maximal orders; this is analogous to considering normal
varieties.

In section \ref{section-numericalrationality}, we define numerical
rationality (c.f. Definition \ref{numericalrationalitydef}) by mimicking
Artin's numerical condition for rational singularities (c.f.
{\cite{artinrational}}, Proposition 1), which states that $Z$ is a rational
singularity if and only if for some resolution $\tilde{Z} \longrightarrow Z$,
every exceptional effective divisor on $\tilde{Z}$ has positive Euler
characteristic. In the noncommutative generalisation, we are led to consider
the Euler characteristic of $\tilde{A}$ restricted to divisors, and arrive at
a similar condition. We show in Proposition \ref{stabilitynum} that numerical
rationality does not depend on the choice of resolution.

Our main result, found in section
\ref{section-logterminalimpliesnumericallyrational}, is that log terminal
orders are numerically rational (c.f. Theorem
\ref{logterminalimpliesnumericallyrational}). The proof of this theorem
depends on an analysis of the dual resolution graphs of the minimal
resolutions of log terminal singularities. The main tool we use is a formula
which computes the Euler characteristic $\chi ( \tilde{A} \otimes \OO_E )$
where $\tilde{A}$ is a terminal order on $\tilde{Z}$ and $E$ is a divisor
whose underlying variety is projective (c.f Theorem \ref{orderadjunction}). We
learned this from an unpublished manuscript, {\cite{artindejong}}. of M. Artin
and A. J. de Jong. We find that $\chi ( \tilde{A} \otimes \OO_E )$ has a nice
expression which resembles the adjunction formula for $\chi ( \OO_E )$, thus
reducing the computation of $\chi ( \tilde{A} \otimes \OO_E )$ to intersection
theory on $\tilde{Z}$. A proof of this is provided in Section
\ref{section-adjunctionformulafororders}.

\subsection{Definitions}

We set up some basic assumptions which will be in force throughout the paper
and briefly review the essential ingredients from the theory of orders on
surfaces. Let $k$ be an algebraically closed field of characteristic zero. All
objects below will be defined over $k$. We assume, once and for all, that all
orders are normal (c.f. {\cite{CI1}}, Definition 2.3). We will not require the
precise definition for normal orders, it is a technical condition arising from
the fact that {\e}tale localisations of a maximal order are in general not
maximal. The normality criterion is a relaxation of maximality which is stable
under {\e}tale localisations. The important point for us is that the
noncommutative Mori program is carried out for normal orders, and in
particular, resolution of singularities for such orders exist.

The most useful invariant for an order in this paper is its canonical divisor.
Let $A$ be an order on a normal surface $Z$ and $D$ be an irreducible curve on
$Z$. We denote by $A_D$ the localisation of $A$ at $D$ and $J ( A_D )$ its
Jacobson radical. The centre $Z ( A_D / J ( A_D ) )$ of $A_D / J ( A_D )$ is a
product of field extensions of $k ( D )$. Define $e_D = \dim_{k ( D )} Z ( A_D
/ J ( A_D ) )$ the ramification index of $A$ at $D$. Also define the
ramification divisor $\Delta_A$ to be the $\mathbbm{Q}$-divisor $\sum_{D \in
Z^1} \left( 1 - 1 / e_D \right) D$ and the canonical divisor $K_A$ of $A$ by
$K_Z + \Delta_A$. There is a finer invariant for an order called the
ramification data, where instead of just remembering the numbers $e_D$, it
remembers the extension $Z ( A_D / J ( A_D ) )$ of $k ( D )$.

\subsection{Acknowledgements}

I take this opportunity to thank my teacher Daniel Chan, who introduced me to
noncommutative algebraic geometry. He asked me to find out whether there is a
good theory of rational singularities for orders, and this paper grew out of
that investigation. I also thank Colin Ingalls for sending me a preprint on
log terminal orders.

\section{\label{section-numericalrationality}Numerical Rationality}

The aim of this section is to define and study a generalisation of rational
singularities for orders. We will outline some of the difficulties involved in
extending such a notion noncommutatively, and hopefully convince the reader
that our proposed definition is interesting. Many naturally occurring
singularities in birational geometry are rational. Recall that log terminal
surface singularities are simply quotients of $\mathbbm{A}^2$ by a finite
subgroup of $G L_2$, and these are rational singularities. Our point of view
is that a version of rational singularities for orders should include the log
terminal orders arising from the noncommutative minimal model program of
{\cite{CI1}}.

The definition for rational singularity for varieties makes essential use of
the existence of a resolution of singularities $\sigma : \tilde{Z}
\longrightarrow Z$. A resolution $\sigma : \tilde{Z} \longrightarrow Z$ is
rational if $\sigma_{\ast} \OO_{\tilde{Z}} = \OO_Z$ and $R^i \sigma_{\ast}
\OO_{\tilde{Z}} = 0$ for $i > 0$; and $Z$ has rational singularities if there
exists a rational resolution $\sigma : \tilde{Z} \longrightarrow Z$. This
definition does not depend on the resolution, since if $Z$ has a rational
resolution, then all resolutions of $Z$ are rational. Orders on surfaces have
resolutions of singularities (c.f. {\cite{CI1}}, Corollary 3.6); a resolution
of an order $A$ consists of a pair $( \sigma : \tilde{Z} \longrightarrow Z,
\tilde{A} )$, where $\sigma : \tilde{Z} \longrightarrow Z$ is a resolution of
varieties and $\tilde{A}$ is a terminal order on $\tilde{Z}$ with
$\sigma^{\ast} A \subset \tilde{A}$ satisfying
\begin{enumerate}
  \item for any exceptional curve $E$ of $\sigma$, the
  $\mathcal{O}_{\tilde{Z}, E}$-order $\tilde{A}_E$ is maximal
  
  \item for any non-exceptional curve $D$ on $\tilde{Z}$, we have $(
  \sigma^{\ast} A )_D = \tilde{A}_D$.
\end{enumerate}
Note that the terminal order $\tilde{A}$ is not unique, since in general, we
can choose different maximal orders $\tilde{A}_E$ containing $( \sigma^{\ast}
A )_E$ for each $E$, and every such choice $\{ \tilde{A}_E \}_E$ produces a
bona fide resolution of $A$ on $\tilde{Z}$. However, by the Artin-Mumford
sequence, the ramification data of $A$ determines the ramification data of its
resolutions (c.f. {\cite{CI1}}, Lemma 3.4). In particular, the ramification
data of $\tilde{A}$ is independent of the choices of maximal orders at
exceptional curves. Armed with this technology, we can explore what it means
for an order to have rational singularities.

The most na{\ii}ve procedure is to replace $\OO_{\tilde{Z}}$ by $\tilde{A}$
and say that $( \sigma, \tilde{A} )$ is a rational resolution if
$\sigma_{\ast} \tilde{A} = A$ and $R^i \sigma_{\ast} \tilde{A} = 0$ for $i >
0$. We see easily that this runs into problems. Firstly, as the following
example shows, such a definition depends on the choice of maximal orders in
blowing up, hence is not Morita invariant (c.f. {\cite{ChanMcKay}},
Proposition 4.1). From the point of view of noncommutative geometry, this is
rather discouraging.

\begin{example}
  \label{negativeexample}Let $Z = \tmop{Spec} k [ [ u, v ] ]$ and
  \begin{eqnarray*}
    A & = & \left(\begin{array}{cc}
      \OO_Z & \OO_Z\\
      ( u^3 - v^2 ) \OO_Z & \OO_Z
    \end{array}\right)
  \end{eqnarray*}
  be a canonical order of type $B L_1$. The order $A$ is ramified on the curve
  $D$ defined by the equation $u^3 - v^2$ and can be resolved by a single
  blowup $\sigma : \tilde{Z} \longrightarrow Z$ at the cusp $p$ of $D$ (c.f.
  {\cite{CI2}}, Figure 1). Let $\tilde{D}$ denote the strict transform of $D$
  and $E$ be the exceptional curve of $\sigma$. There are three non-isomorphic
  terminal orders on $\tilde{Z}$
  \begin{eqnarray*}
    \tilde{A}_m & = & \left(\begin{array}{cc}
      \OO_{\tilde{Z}} & \OO_{\tilde{Z}} ( m E )\\
      \OO_{\tilde{Z}} ( - \tilde{D} - m E ) & \OO_{\tilde{Z}}
    \end{array}\right)
  \end{eqnarray*}
  for $m = 0, 1, 2$ which are maximal at $E$ and contain $\sigma^{\ast} A$, so
  $R^1 \sigma_{\ast} \tilde{A}_m$ vanishes if and only if $m \neq 2$. 
\end{example}

Moreover, the same example shows that there exists a resolution of a canonical
order that is not rational in the na{\ii}ve sense above. This transgresses our
requirement that the canonical orders (which are log terminal) of the
noncommutative Mori program should be rational. An alternative is to use the
dual formulation of the definition, that is say that a resolution $( \sigma,
\tilde{A} )$ is rational if $\omega_A$ is a Cohen-Macaulay sheaf and
$\sigma_{\ast} \omega_{\tilde{A}} = \omega_A$. Unfortunately, this too is
susceptible to the same objections as above. We note here that due to the
absence of a Grauert-Riemenschneider vanishing theorem, the above two
formulations for rational resolutions for orders are not equivalent.

To get a good notion for rational resolutions for orders, we generalise
Artin's numerical criterion for rational singularities on varieties, which
states that $Z$ has rational singularities if and only if for some resolution
$\sigma : \tilde{Z} \longrightarrow Z$, we have $\chi ( \OO_E ) > 0$ for all
exceptional divisors $E > 0$ on $\tilde{Z}$.

\begin{definition}
  \label{numericalrationalitydef}Let $A$ be a normal order on a surface $Z$. A
  resolution $( \sigma : \tilde{Z} \longrightarrow Z, \tilde{A} )$ of $A$ is
  numerically rational if $\chi ( \tilde{A} \otimes \OO_E ) > 0$ holds for all
  exceptional divisors $E > 0$ on $\tilde{Z}$. The order $A$ is numerically
  rational if every resolution is numerically rational.
\end{definition}

We prove below that the above definition has the nice property that if
numerical rationality holds for some resolution, then it holds for all
resolutions. This generalises the corresponding fact for rational resolutions
for varieties. The adjunction formula proved in Theorem \ref{orderadjunction}
will be used in the proof of the next proposition, and we state the result for
the reader's convenience: let $\tilde{A}$ be a terminal order of rank $n^2$ on
$\tilde{Z}$ and $E$ be an effective exceptional divisor, then
\begin{eqnarray*}
  \chi ( \tilde{A} \otimes_{\tilde{Z}} \OO_E ) & = & - \frac{n^2}{2} (
  K_{\tilde{A}} + E ) E.
\end{eqnarray*}
We see immediately that $\chi ( \tilde{A} \otimes_{\tilde{Z}} \OO_E )$ depends
only on the ramification data. In particular, if $\sigma$ is a rational
resolution, then numerical rationality is a Morita invariant property (c.f.
{\cite{ChanMcKay}}, Proposition 4.1).

\begin{proposition}
  \label{stabilitynum}Let $( \sigma : \tilde{Z} \longrightarrow Z, \tilde{A}
  )$, $( \tau : \tilde{Y} \longrightarrow Z, \tilde{B} )$ be resolutions of
  $A$ and suppose $\tau = \sigma \beta$ where $\beta$ is a blowup at a point
  $p \in \tilde{Z}$. Then $( \sigma, \tilde{A} )$ is numerically rational if
  and only if $( \tau, \tilde{B} )$ is numerically rational. 
\end{proposition}

\begin{proof}
  Since both $\tilde{A}$ and $\tilde{B}$ are terminal orders, we can use
  Theorem \ref{orderadjunction} to compute their Euler characteristics when
  restricted to divisors. Let $R$ denote the ramification divisor of
  $\tilde{A}$ on $\tilde{Z}$. We have two cases to consider, depending on
  whether $p$ belongs to the singular locus of $\tmop{supp} R$.
  
  We denote by $E_0 = \tmop{Ex} ( \beta )$ the exceptional curve on
  $\tilde{Y}$ contracted by $\beta$. If $p$ is not in the singular locus of
  $\tmop{supp} R$, we see that $E_0$ is unramified. If, in addition, $p \in
  \tmop{supp} R$ then $\beta^{\ast} \Delta_{\tilde{A}} - \left( 1 - 1 / e_1
  \right) E_0 = \Delta_{\tilde{B}}$ where $e_1$ is the ramification index of
  the irreducible component of $\tmop{supp} R$ containing $p$. If $p \nin
  \tmop{supp} R$, then $\beta^{\ast} \Delta_{\tilde{A}} = \Delta_{\tilde{B}}$.
  For convenience, we will write this as $\beta^{\ast} \Delta_{\tilde{A}} - (
  1 - 1 / e_1 ) E_0 = \Delta_{\tilde{B}}$ for $e_1 = 1$.
  
  Now suppose $p \in ( \tmop{supp} R )_{\tmop{sing}}$. Note that since
  $\tilde{A}$ is terminal, $R$ only has nodal singularities. Let $R_1, R_2$ be
  irreducible components of $R$ intersecting transversely at $p$, and denote
  by $e_i$ the ramification index of $R_i$. Then $e_1 = s e_2$ for some
  integer $s$. The Artin-Mumford sequence can be used to show that $\tilde{A}$
  is totally ramified at $E_0$ with $e_0 = e_2$ (c.f. Lemma 3.4,
  {\cite{CI1}}). Hence $\beta^{\ast} \Delta_{\tilde{A}} - \left( 1 - 1 / e_1
  \right) E_0 = \Delta_{\tilde{B}}$.
  
  We can write a general effective divisor $E$ on $\tilde{Y}$ as
  $\beta^{\ast} \tilde{E} + m E_0$ where $\tilde{E}$ is some effective divisor
  on $\tilde{Z}$ and $m \in \mathbbm{Z}$. Since $\beta$ is the blowup of a
  smooth point, we know that $\beta^{\ast} K_{\tilde{Z}} + E_0 =
  K_{\tilde{Y}}$. In each case, we get the following,
  \begin{eqnarray*}
    \chi ( \tilde{B} \otimes_{\tilde{Y}} \mathcal{O}_E ) & = & \chi (
    \tilde{A} \otimes_{\tilde{Z}} \mathcal{O}_{\tilde{E}} ) + \frac{n^2}{2} m
    \left( m + \frac{1}{e_1} \right)
  \end{eqnarray*}
  for some $e_1 > 0$. If $\chi ( \tilde{B} \otimes_{\tilde{Y}} \mathcal{O}_E )
  > 0$ for all $E > 0$, then putting $m = 0$ gives $\chi ( \tilde{A}
  \otimes_{\tilde{Z}} \mathcal{O}_{\tilde{E}} ) > 0$ for all $\tilde{E} > 0$.
  Conversely, if $\chi ( \tilde{A} \otimes_{\tilde{Z}} \mathcal{O}_{\tilde{E}}
  )$ is positive for all $\tilde{E} > 0$, then we can conclude that $\chi (
  \tilde{B} \otimes_{\tilde{Y}} \mathcal{O}_E ) \geqslant 0$ for all
  $\tilde{E} > 0$ and $m \in \mathbbm{Z}$. To see that $\chi ( \tilde{B}
  \otimes_{\tilde{Y}} \mathcal{O}_E ) > 0$, we find that the only nontrivial
  solution of $\chi ( \tilde{B} \otimes_{\tilde{Y}} \mathcal{O}_E ) = 0$
  occurs when $e_1 = 1$, $\tilde{E} = 0$ and $m = - 1$. This does not
  correspond to an effective divisor on $\tilde{Y}$. Hence $\chi ( \tilde{B}
  \otimes_{\tilde{Y}} \mathcal{O}_E ) > 0$ for all $E > 0$. 
\end{proof}

We say that a resolution $( \sigma : \tilde{Z} \longrightarrow Z, \tilde{A} )$
of an order $A$ is minimal if the canonical divisor $K_{\tilde{A}}$ is
$\sigma$-nef (c.f {\cite{CI1}}, Theorem 3.10). 

\begin{corollary}
  \label{onlyminimalneeded}An order $A$ is numerically rational if and only if
  any resolution is numerically rational.
\end{corollary}

\begin{proof}
  Suppose $( \sigma, \tilde{A} )$ is a resolution of $A$. If $( \sigma,
  \tilde{A} )$ is not minimal, then there exists a $K_{\tilde{A}}$-negative
  curve $E$ with $E^2 < 0$. By {\cite{CI1}}, Theorem 3.10, we can factor
  $\sigma = \tau' \beta'$ through a blowup $\beta'$ at a point which contracts
  $E$, and there exists a terminal order $A'_1$ such that $( \tau', A'_1 )$ is
  a resolution of $A$. The terminal order $A_1'$ is obtained by taking the
  reflexive hull of $\beta'_{\ast} \tilde{A}$. Repeating this until we reach a
  minimal resolution allows us to factor $\sigma = \tau \beta$ where $\beta$
  is a sequence of blowups centred at closed points, and obtain a terminal
  order $A_1$ such that $( \tau, A_1 )$ is a minimal resolution of $A$.
  
  By Proposition \ref{stabilitynum}, $( \sigma, \tilde{A} )$ is numerically
  rational if and only if $( \tau, A_1 )$ is numerically rational. According
  to Theorem 2.15 of {\cite{CI2}}, minimal resolutions of $A$ have the same
  centres and ramification data. Since numerical rationality depends only on
  the ramification data, the result follows.
\end{proof}

Recall that canonical orders have crepant minimal resolutions (c.f.
{\cite{CI2}}, Proposition 6.1), that is, if $( \sigma, \tilde{A} )$ is a
minimal resolution of the canonical order $A$, then $K_{\tilde{A}} =
\sigma^{\ast} K_A$. It is easy to show that canonical orders are numerically
rational.

\begin{corollary}
  \label{canonicalimpliesnumericallyrational}Let $( \sigma : \tilde{Z}
  \longrightarrow Z, \tilde{A} )$ be a crepant resolution of the $\OO_Z$-order
  $A$. Then $( \sigma, \tilde{A} )$ is a numerically rational resolution. In
  particular, canonical orders are numerically rational. 
\end{corollary}

\begin{proof}
  If $( \sigma, \tilde{A} )$ is crepant, then $\chi ( \tilde{A}
  \otimes_{\tilde{Z}} \OO_E ) = - n^2 E^2 / 2$, which is positive for any
  exceptional divisor $E$, so $( \sigma, \tilde{A} )$ is numerically rational.
  Let $A$ be a canonical order. The minimal resolution of a canonical order
  $A$ is crepant, hence $A$ is numerically rational.
\end{proof}

Our definition of numerical rationality is weaker than the na{\ii}ve
generalisation of rational resolutions to orders: if $R^1 \sigma_{\ast}
\tilde{A} = 0$, then we can see by taking the long exact sequence in
cohomology associated to the short exact sequence
\[ 0 \longrightarrow \tilde{A} ( - E ) \longrightarrow \tilde{A}
   \longrightarrow \tilde{A} \otimes_{\tilde{Z}} \OO_E \longrightarrow 0 \]
that $h^1 ( \tilde{A} \otimes_{\tilde{Z}} \OO_E ) = 0$. Since $\OO_E \subset
\tilde{A} \otimes_{\tilde{Z}} \OO_E$, we have $h^0 ( \tilde{A}
\otimes_{\tilde{Z}} \OO_E ) > 0$ hence $\chi ( \tilde{A} \otimes_{\tilde{Z}}
\OO_E ) > 0$. Example \ref{negativeexample} and Corollary
\ref{canonicalimpliesnumericallyrational} shows that numerical rationality is
strictly weaker than the na{\ii}ve generalisation of rationality.

We conclude this section with an example of an order which is not numerically
rational.

\begin{example}
  Consider the simple elliptic singularity of type $\tilde{E}_6$, which is
  given by the equation $f_{\lambda} = 0$ where $f_{\lambda} = u^3 + v^3 + w^3
  + \lambda u v w$ (c.f. {\cite{dimca}}, (4.9)) for some $\lambda \in k$. We
  construct below an order $A$ with centre $Z = \tmop{Spec} k [ [ u, v, w ] ]
  / ( f_{\lambda} )$ whose minimal resolution is not numerically rational. Let
  $A$ be the $k [ [ u, v, w ] ] / ( f_{\lambda} )$-algebra generated by $x, y$
  with relations $x^2 = u$, $y^2 = v$ and $x y + y x = 0$. Then $A$ is a
  maximal order of rank $4$ over $Z$.
  
  Let $\sigma : \tilde{Z} \longrightarrow Z$ be the minimal resolution of
  $Z$, then $E = \tmop{Ex} ( \sigma )$ is an elliptic curve with $E^2 = - 3$.
  Let $( \sigma, \tilde{A} )$ be a blowup of $A$. Since $A$ is maximal,
  $\tilde{A}$ is a maximal order on $\tilde{Z}$ containing $\sigma^{\ast} A$.
  A local computation shows that $( \sigma^{\ast} A )_E$ is contained in a
  unique maximal order, hence $\tilde{A}$ is the unique blowup of $A$ along
  $\sigma$. We can describe $\tilde{A}$ as follows: on the open affine set $U
  = \tmop{Spec} k [ [ u, v, w ] ] [ v / u, w / u ] / ( f_{\lambda} u^{- 3} )$,
  we have
  \begin{eqnarray*}
    \tilde{A} ( U ) & = & \sigma^{\ast} A ( U ) \left\langle x y u^{- 1}
    \right\rangle
  \end{eqnarray*}
  and similarly for the other standard open affine sets $V = \tmop{Spec} k [ [
  u, v, w ] ] [ u / v, w / v ] / ( f_{\lambda} )$ and $W = \tmop{Spec} k [ [
  u, v, w ] ] [ u / w, v / w ] / ( f_{\lambda} )$ of $\tilde{Z}$. One can
  check that $\tilde{A}$ is a terminal order ramified on $E$ and two divisors
  $D_1, D_2$ transverse to $E$, each with ramification index $2$. The
  equations for $D_1$ and $D_2$ on $W$ are $u / w = 0$ and $v / w = 0$.
  
  We show that the resolution $( \sigma, \tilde{A} )$ is not numerically
  rational. The simple elliptic singularity is log canonical, so
  $K_{\tilde{Z}} = \sigma^{\ast} K_Z - E$. This gives $K_{\tilde{A}} =
  \sigma^{\ast} K_Z - E / 2 + D / 2$ so by the adjunction formula for orders,
  we have
  \begin{eqnarray*}
    \chi ( \tilde{A} \otimes_{\tilde{Z}} \OO_{m E} ) & = & 3 m \left( 2 m - 3
    \right)
  \end{eqnarray*}
  which is negative for $m = 1$. 
\end{example}

\section{\label{section-logterminalimpliesnumericallyrational}Log terminal
implies numerically rational}

In this section, we prove a noncommutative version of the following result:
log terminal singularities are rational singularities. There is a notion of
log terminal orders developed in the context of the noncommutative Mori theory
of {\cite{CI1}}, and the analogue for rational singularities is provided by
our notion of numerical rationality (c.f. Definition
\ref{numericalrationalitydef}). The noncommutative version of the above result
has the following pleasant statement.

\begin{theorem}
  \label{logterminalimpliesnumericallyrational}If $A$ is a log terminal order
  on $Z$, then $A$ is numerically rational.
\end{theorem}

Note that if $A$ is log terminal, the associated log pair $( Z, \Delta_A )$ of
$A$ is klt ({\cite{CI1}}, Proposition 3.15). It follows from {\cite{KoMo}},
Corollary 2.35 that $Z$ has log terminal singularities. So we assume below
that $Z$ is the spectrum of a local ring with log terminal singularities. To
prove that log terminal orders are numerically rational, by Corollary
\ref{onlyminimalneeded} we need only study their minimal resolutions. Since
all minimal resolutions have the same ramification data ({\cite{CI2}}, Theorem
2.5) we will use the following characterisation of log terminal orders, which
is equivalent to the definition in {\cite{CI1}}. Let $A$ be an order on $Z$
and $( \sigma : Z' \longrightarrow Z, A' )$ be any minimal resolution. We can
write
\begin{eqnarray*}
  K_{A'} & = & \sigma^{\ast} K_A + \sum_i a_i E_i
\end{eqnarray*}
where the $E_i$'s range over the exceptional curves on $Z'$. Then $A$ is log
terminal if and only if $\min \{ a_i e_i \} > - 1$, where $e_i$ is the
ramification index of $A'$ at $E_i$.

Let $A$ be a log terminal order on $Z$ and $( \sigma : \tilde{Z}
\longrightarrow Z, \tilde{A} )$ be any resolution. Denote by $\tmmathbf{E}$
the subgroup $\tmop{Pic} \tilde{Z}$ generated by the exceptional curves. The
intersection product on exceptional curves is well defined, and it endows
$\tmmathbf{E}$ with the structure of a quadratic $\mathbbm{Z}$-module. Given a
$D \in \tmmathbf{E} \otimes_{\mathbbm{Z}} \mathbbm{Q}$, we define the function
$f_{\sigma, D} : \tmmathbf{E} \longrightarrow \mathbbm{Q}$ by $E \longmapsto -
( D + E ) E$. Let $\tmmathbf{E}^+ = \{ \sum a_i E_i \in \tmmathbf{E} \mid a_i
\geqslant 0 \}$ denote the effective cone of $\tmmathbf{E}$. By Theorem
\ref{orderadjunction}, $( \sigma, \tilde{A} )$ is a numerically rational
resolution if and only if $f_{\sigma, K_{\tilde{A}}} ( E ) > 0$ for all $E \in
\tmmathbf{E}^+ \backslash \{ 0 \}$.

The function $f_{\sigma, K_{\tilde{A}}}$ is the sum of a positive definite
quadratic form $q ( E ) = - E^2$ and a linear function $\ell ( E ) = - D E$ on
$\tmmathbf{E}$. The log terminal condition on $K_{\tilde{A}}$ puts constraints
on the coefficients of $\ell$. To get some information out of these
constraints, it is profitable to choose a different $\mathbbm{Z}$-basis $\{
N_i \}$ for $\tmmathbf{E}$. We define $N_i$ as follows: for an exceptional
curve $E_i$ on $\tilde{Z}$, there is a unique factorisation $\tilde{Z}
\longrightarrowlim^{\tau_i} Z' \longrightarrowlim^{\tau'_i} Z$ of $\sigma$
satisfying the following properties
\begin{enumerate}
  \item $Z'$ is smooth,
  
  \item $E_i$ is not contracted by $\tau_i$, so $\tau_{i \ast} E_i$ is a curve
  on $Z'$, and
  
  \item there are no $( - 1 )$-curves on $Z'$ except for possibly $\tau_{i
  \ast} E_i$.
\end{enumerate}
Let $N_i = \tau^{\ast}_i \tau_{i \ast} E_i$. It is easy to see that one
obtains the factorisation above by sequentially contracting $( - 1 )$-curves
except for the pushforwards of $E_i$.

We will adopt the following notation for the exceptional curves on
$\tilde{Z}$: the resolution $\sigma$ factors through a minimal resolution
$\pi_0 : Z_0 \longrightarrow Z$ of $Z$ so that $\sigma = \pi_0 \pi$, we denote
by $E_1, \ldots, E_r$ the exceptional curves not contracted by $\pi$ and label
the rest by $E_{r + 1}, \ldots, E_{r + \ell}$.

\begin{proposition}
  \label{diagonalisationlemma}Let $( \sigma : \tilde{Z} \longrightarrow Z,
  \tilde{A} )$ be a minimal resolution of an order $A$ on $Z$ and let $\sigma
  = \pi_0 \pi$ be as above. Then
  \begin{eqnarray}
    f_{\sigma, K_{\tilde{A}}} ( E ) & = & f_{\pi_0, \pi_{\ast} K_{\tilde{A}}}
    ( \pi_{\ast} E ) + \left( \sum_{j = 1}^{\ell} \left( N_{r + j} E \right)
    N_{r + j} \right) ( E + K_{\tilde{A}} ) .  \label{diagonalisationeq}
  \end{eqnarray}
\end{proposition}

\begin{proof}
  Since $f_{\pi_0, \pi_{\ast} K_{\tilde{A}}} ( \pi_{\ast} E ) = - \pi^{\ast}
  \pi_{\ast} E ( E + K_{\tilde{A}} )$, it suffices to show that
  \begin{eqnarray}
    E & = & \pi^{\ast} \pi_{\ast} E - \left( \sum_{j = 1}^{\ell} \left( N_{r +
    j} E \right) N_{r + j} \right) .  \label{divisoreq}
  \end{eqnarray}
  A simple computation gives
  \begin{eqnarray*}
    N_{r + j} E_i & = & \left\{\begin{array}{ll}
      - 1 & \text{if } i = r + j\\
      1 & \text{if } \tau_i = \alpha \tau_{r + j} \text{ where } \alpha \text{
      is a blowup centered at a point on } \tau_{i \ast} E_i\\
      0 & \text{otherwise}
    \end{array}\right. .
  \end{eqnarray*}
  Let $E_{r + j_1}, \ldots, E_{r + j_s}$ be components of $N_i$ which
  intersect $E_i$, then $E_i = N_i - N_{r + j_1} - \cdots - N_{r + j_s}$. The
  curve $E_{r + j_i}$ is contracted by $\tau_i$, hence we can factor $\tau_i =
  \alpha \tau_{r + j}$ for some birational map $\alpha$. Since $E_{r + j_i}
  E_i \neq 0$, $\alpha$ must be a single blowup centered at a point on
  $\tau_{i \ast} E_i$. Hence (\ref{divisoreq}) follows from the above
  computation for $N_{r + j} E_i$. 
\end{proof}

We can deduce from (\ref{diagonalisationeq}) that the inequalities $N_{r + j}
K_{\tilde{A}} < 1$ for $j = 1, \ldots, \ell$ are necessary conditions for $(
\sigma, \tilde{A} )$ to be numerically rational, since if $N_{r + j}
K_{\tilde{A}} \geqslant 1$, we have
\begin{eqnarray*}
  f_{\sigma, K_{\tilde{A}}} ( N_{r + j} ) & = & - ( N_{r + j} K_{\tilde{A}} -
  1 ) \leqslant 0.
\end{eqnarray*}
\begin{proposition}
  \label{smoothterms}Suppose $A$ is log terminal. Then $0 \leqslant N_{r + j}
  K_{\tilde{A}} < 1$ for $j = 1, \ldots, \ell$.
\end{proposition}

\begin{proof}
  Since $( \sigma, \tilde{A} )$ is a minimal resolution, $K_{\tilde{A}}$ is
  $\sigma$-nef (c.f. {\cite{CI1}}, Theorem 3.10). Since $N_{r + j}$ is
  effective, we get the first inequality. Note that $N_{r + j} E_{r + j} = (
  \tau_{r + j \ast} E_{r + j} )^2 = - 1$ and $N_{r + j} E_k = 0$ for any other
  exceptional curve $E_k \subseteq \tmop{supp} N_{r + j}$. Clearly the
  intersection numbers of $N_{r + j}$ with exceptional curves away from its
  support are non-negative, in fact, if $E_i$ is not an irreducible component
  of $N_{r + j}$, then $E_i N_{r + j} = 0$ or $1$. This gives
  \begin{eqnarray*}
    N_{r + j} K_{\tilde{A}} & = & - a_{r + j} + \sum_i a_i
  \end{eqnarray*}
  where the summation ranges over all $i$ where $E_i$ intersects $N_{r + j}$.
  Now since $K_{\tilde{A}}$ is $\sigma$-nef, we have by {\cite{KoMo}}, Lemma
  3.41 that $a_i \leqslant 0$ for all $i$. Moreover, since $A$ is log
  terminal, we have $a_{r + j} > - 1$, hence $N_{r + j} K_{\tilde{A}} < 1$.
\end{proof}

The above propositions shows that log terminal orders with smooth centres are
numerically rational, and that in general, a minimal resolution $( \sigma,
\tilde{A} )$ of a log terminal order is numerically rational if and only if
$f_{\pi_0, \pi_{\ast} K_{\tilde{A}}} ( \pi_{\ast} E ) > 0$ for all $E \in
\tmmathbf{E}^+ \backslash \{ 0 \}$ such that $\pi_{\ast} E > 0$. This allows
us to work directly with the minimal resolution $\pi_0 : Z_0 \longrightarrow
Z$. Henceforth, we will drop the $\pi_{\ast}$ and refer to the exceptional
curves on $Z_0$ by $E_1, \ldots, E_r$ and denote by $\tmmathbf{E}_0 \subseteq
\tmop{Pic} Z_0$ the subgroup generated by $E_1, \ldots, E_r$.

Recall that the numerical cycle $Z_{\tmop{num}}$ with respect to the
birational morphism $\pi_0 : Z_0 \longrightarrow Z$ is defined to be the
minimal effective exceptional divisor $E$ on $Z_0$ such that $- E$ is
$\pi_0$-nef. Given a connected effective exceptional divisor $D$, we can
contract $\tmop{supp} D$ to get a birational morphism $\pi_D : Z_0
\longrightarrow Z_D$. We define $D_{\tmop{num}}$ to be the numerical cycle
with respect to Ÿ$\pi_D$, and call it the numerical cycle of the support of
$D$. We call a connected effective exceptional divisor $D$ on $Z_0$ special if
$D = D_{\tmop{num}}$. In particular, the numerical cycle $Z_{\tmop{num}}$ of
$Z_0$ is a special divisor.

As we shall prove in Theorem \ref{technicaltheorem}, $f_{\pi_0, \pi_{\ast}
K_{\tilde{A}}}$ is positive for all $E \in \tmmathbf{E}^+_0 \backslash \{ 0
\}$ if its values at special divisors are positive. Theorem
\ref{logterminalimpliesnumericallyrational} follows from the next two results.

\begin{proposition}
  \label{logterminalordersandspecialdivisors}Let $A$ be a log terminal order
  on $Z$. Then $f_{\pi_0, \pi_{\ast} K_{\tilde{A}}} ( E ) > 0$ for all special
  divisors $E \in \tmmathbf{E}_0$.
\end{proposition}

\begin{proof}
  Let $E$ be a special divisor. Suppose $E_j$ is not contained in $\tmop{supp}
  E$, then $E_{\tmop{num}} E_j \geqslant 0$. Since $a_i \leqslant 0$ (c.f.
  proof of Proposition \ref{smoothterms}), we have
  \begin{eqnarray*}
    f_{\pi_0, \pi_{\ast} K_{\tilde{A}}} ( E_{\tmop{num}} ) & \geqslant & -
    E_{\tmop{num}} \left( E_{\tmop{num}} + \sum_{E_i \subseteq \tmop{supp} E}
    a_i E_i \right) .
  \end{eqnarray*}
  Now $A$ is log terminal, so $a_i > - 1$ for all $i$. Moreover, by the
  definition of $E_{\tmop{num}}$, we have $- E_{\tmop{num}} E_i \geqslant 0$
  for any $E_i \subseteq \tmop{supp} E$, so
  \begin{eqnarray*}
    f_{\pi_0, \pi_{\ast} K_{\tilde{A}}} ( E_{\tmop{num}} ) & > & -
    E_{\tmop{num}} ( E_{\tmop{num}} - E_{\tmop{red}} ),
  \end{eqnarray*}
  where $E_{\tmop{red}}$ denotes the reduced exceptional divisor with the same
  support as $E_{\tmop{num}}$. Since $E_{\tmop{num}}$ is the numerical cycle
  of its support, we see that $E_{\tmop{num}} - E_{\tmop{red}}$ is an
  effective divisor with support contained in $\tmop{supp} E$. This gives
  $f_{\pi_0, \pi_{\ast} K_{\tilde{A}}} ( E_{\tmop{num}} ) > 0$.
\end{proof}

\begin{theorem}
  \label{technicaltheorem}Let $\pi_0 : Z_0 \longrightarrow Z$ be the minimal
  resolution of a log terminal singularity and $g : \tmmathbf{E}
  \longrightarrow \mathbbm{Q}$ be a function $g ( E ) = - E^2 + \ell ( E )$
  with $\ell$ linear and $\ell ( E_i ) \leqslant 0$ for $i = 1, \ldots, r$.
  Then $g ( E ) > 0$ for all $E \in \tmmathbf{E}^+_0 \backslash \{ 0 \}$ if
  and only if $g ( E ) > 0$ for all special $E \in \tmmathbf{E}_0$.
\end{theorem}

\begin{note}
  Note that $f_{\pi_0, \pi_{\ast} K_{\tilde{A}}}$ satisfies the above
  hypotheses for $g$ if $( \sigma, \tilde{A} )$ is a minimal resolution. Since
  in this case $K_{\tilde{A}}$ is $\sigma$-nef, we have $\pi_{\ast}
  K_{\tilde{A}} \pi_{\ast} E_i = K_{\tilde{A}} N_i \geqslant 0$ for $i = 1,
  \ldots, r$, hence $\pi_{\ast} K_{\tilde{A}}$ is $\pi_0$-nef.
\end{note}

The rest of this section is devoted to the proof of the above theorem. We fix
notation for the rest of the section: let $\pi_0 : Z_0 \longrightarrow Z$ be
the minimal resolution of a log terminal singularity. We denote by $E_1,
\ldots, E_r$ the exceptional curves and $Z_{\tmop{num}}$ the numerical cycle
on $Z_0$. Also we will denote by $g ( E ) = - E^2 + \ell ( E )$ a function
$\tmmathbf{E}_0 \longrightarrow \mathbbm{Q}$ satisfying the hypothesis of
Theorem \ref{technicaltheorem}.

\subsection{Modified numerical cycle}

The usual notion of numerical cycle can be modified with respect to a given
effective exceptional divisor $D$ as follows: we define the numerical cycle
$D'$ associated to $D$ to be the minimal effective divisor satisfying $D
\leqslant D'$ and $- D'$ is $\pi_0$-nef. When $D = 0$, then $D'$ is just the
usual numerical cycle $Z_{\tmop{num}}$. To see that $D'$ exists and is unique
for a given $D$, pick an integer $n$ such that $D \leqslant n Z_{\tmop{num}}$.
Then the divisor $D'' = \gcd \{ C \mid D \leqslant C \leqslant n
Z_{\tmop{num}}, - C \text{ is } \pi_0 \text{-nef} \}$ is well defined since
the $\gcd$ is taken over finitely many exceptional divisors, and clearly $D' =
D''$. We can construct $D'$ inductively by the following procedure, which is
modelled on the construction of $Z_{\tmop{num}}$ (c.f. {\cite{reid-chapters}},
Section 4.5).

We start with $D_0 = D$ and define $D_{i + 1}$ recursively as follows. If $-
D_i$ is $\pi_0$-nef, then we are done; otherwise there exists some irreducible
exceptional curve $E$ such that $D_i \cdot E > 0$. Define $D_{i + 1} = D_i +
E$ and repeat. The following lemma shows that the above procedure terminates
at $D'$.

\begin{lemma}
  For each $i$, we have $D_i \leqslant D'$.
\end{lemma}

\begin{proof}
  Suppose $D_i \leqslant D'$ and let $E$ be any exceptional curve. If the
  effective divisor $D' - D_i$ is supported away from $E$, then $D_i E
  \leqslant D' E \leqslant 0$. Hence $D_{i + 1} = D_i + \tilde{E}$ where
  $\tilde{E}$ is an exceptional curve whose multiplicity in $D_i$ is strictly
  less than its multiplicity in $D'$. This shows that $D_{i + 1} \leqslant
  D'$.
\end{proof}

The following inequality will be useful for bounding $- D^2$ below.

\begin{lemma}
  \label{firstlemma}Let $D$ and $D'$ be as above. Then $h^0 ( \OO_D )
  \geqslant h^0 ( \OO_{D'} )$.
\end{lemma}

\begin{proof}
  It suffices to show that $h^0 ( \OO_{D_i} ) \geqslant h^0 ( \OO_{D_{i + 1}}
  )$. By construction $D_{i + 1} = D_i + E$ for some exceptional curve $E$
  with $E \cdot D_i > 0$. Applying $\chi$ to the exact sequence $0
  \longrightarrow \OO_E ( - D_i ) \longrightarrow \OO_{D_{i + 1}}
  \longrightarrow \OO_{D_i} \longrightarrow 0$ and using the fact that $\pi_0$
  is a rational resolution, we obtain
  \begin{eqnarray*}
    h^0 ( \OO_{D_{i + 1}} ) & = & \chi ( \OO_E ( - D_i ) ) + h^0 ( \OO_{D_i} )
    .
  \end{eqnarray*}
  Since $E \simeq \mathbbm{P}^1$ and $E \cdot D_i > 0$, we have
  \begin{eqnarray*}
    h^0 ( \OO_{D_{i + 1}} ) & = & 1 - E \cdot D_i + h^0 ( \OO_{D_i} )
    \leqslant h^0 ( \OO_{D_i} ) .
  \end{eqnarray*}
\end{proof}

\subsection{Bounding $- D^2$}

We gather here a few facts about numerical invariants of singularities. Let
$\pi_0 : Z_0 \longrightarrow Z$ be a resolution of a rational surface
singularity. We denote by $b_i = - E_i^2$ for exceptional curves $E_1, \ldots,
E_r$. Recall that its multiplicity can be expressed in terms of the numerical
cycle by the formula $m = - Z^2_{\tmop{num}}$ (c.f. {\cite{reid-chapters}},
section 4.17). If, in addition, $Z$ has log terminal singularities, then $m =
- Z^2_{\tmop{num}}$ simplifies to
\begin{eqnarray}
  m & = & 2 + \sum_{i = 1}^r ( b_i - 2 ) .  \label{multiplicityformula}
\end{eqnarray}
This was observed in {\cite{brieskorn}}, proof of Satz 2.11, and can be
deduced from the following proposition, which we will also need for the proof
of Proposition \ref{decompositionprop}.

\begin{proposition}
  \label{numericalcycleoflogterminalsingularities}Let $\pi_0 : Z_0
  \longrightarrow Z$ be the minimal resolution of a log terminal singularity,
  with exceptional curves $E_1, \ldots, E_r$ and numerical cycle
  $Z_{\tmop{num}}$. If $b_i > 2$, then the multiplicity of $E_i$ in
  $Z_{\tmop{num}}$ is $1$.
\end{proposition}

\begin{proof}
  We refer the reader to {\cite{nikulin}}, Figure 1, for the intersection
  graphs of the exceptional curves on minimal resolutions of log terminal
  singularities. Let $\Gamma$ be such a graph and let $v ( i )$ denote the
  number of edges incident on a vertex $i$. The first observation is if $v ( i
  ) \leqslant b_i$ for all vertices $i$, then the numerical cycle
  $Z_{\tmop{num}}$ is reduced. In particular, the proposition holds for such
  graphs $\Gamma$. The graphs $\Gamma$ for which there exist vertices $i, j$
  such that $v ( i ) < b_i$ and $b_j > 2$ have the following forms
  
  \vspace{0.5cm}
$ \begin{array}[c]{lll}
1.&  
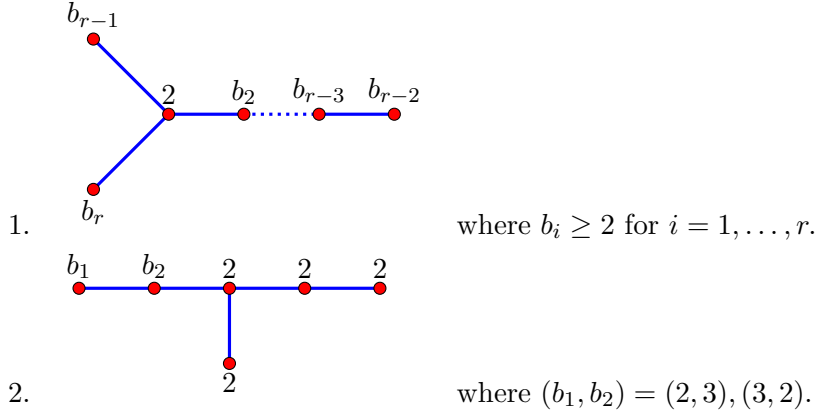
\begin{tikzpicture}  
    \draw[very thick,blue] (-1,-1) -- (0,0);
    \draw[very thick,blue] (-1,1) -- (0,0);
    \draw[very thick,blue] (0,0) -- (1,0);
    \draw[very thick,dotted,blue] (1,0) -- (2,0);
    \draw[very thick,blue] (2,0) -- (3,0);
    \draw[fill=red] (-1,1) circle (0.08) node[above] {$b_{r-1}$};
    \draw[fill=red] (-1,-1) circle (0.08) node[below] {$b_{r}$};
    \draw[fill=red] (0,0) circle (0.08) node[above] {$2$};
    \draw[fill=red] (1,0) circle (0.08) node[above] {$b_2$};
    \draw[fill=red] (2,0) circle (0.08) node[above] {$b_{r-3}$};
    \draw[fill=red] (3,0) circle (0.08) node[above] {$b_{r-2}$};
\end{tikzpicture}
& \text{where } b_i \geq 2 \text{ for } i=1,\ldots,r.\\
2.&  \begin{tikzpicture}  
    \draw[very thick,blue] (-2,0) -- (-1,0);
    \draw[very thick,blue] (-1,0) -- (0,0);
    \draw[very thick,blue] (0,0) -- (1,0);
    \draw[very thick,blue] (1,0) -- (2,0);
    \draw[very thick,blue] (0,0) -- (0,-1);
    \draw[fill=red] (-2,0) circle (0.08) node[above] {$b_1$};
    \draw[fill=red] (-1,0) circle (0.08) node[above] {$b_2$};
    \draw[fill=red] (0,0) circle (0.08) node[above] {$2$};
    \draw[fill=red] (1,0) circle (0.08) node[above] {$2$};
    \draw[fill=red] (2,0) circle (0.08) node[above] {$2$};
    \draw[fill=red] (0,-1) circle (0.08) node[below] {$2$};
  \end{tikzpicture}
& \text{where }(b_1,b_2) = (2,3),(3,2).
\end{array} $\\
\vspace{0.5cm}

  In case 1, let $j$ be the minimal integer such that $b_i = 2$ for all $i <
  j$. Then the numerical cycle is $2 ( E_1 + \cdots + E_{j - 1} ) + E_j +
  \cdots + E_r$. For case 2, the numerical cycles are, respectively,
  
  
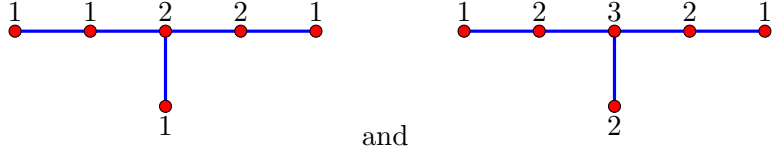
\begin{center}
  \begin{tikzpicture}
    \draw[very thick,blue] (-2,0) -- (-1,0);
    \draw[very thick,blue] (-1,0) -- (0,0);
    \draw[very thick,blue] (0,0) -- (1,0);
    \draw[very thick,blue] (1,0) -- (2,0);
    \draw[very thick,blue] (0,0) -- (0,-1);

    \draw[fill=red] (-2,0) circle (0.08) node[above] {$1$};
    \draw[fill=red] (-1,0) circle (0.08) node[above] {$1$};
    \draw[fill=red] (0,0) circle (0.08) node[above] {$2$};
    \draw[fill=red] (1,0) circle (0.08) node[above] {$2$};
    \draw[fill=red] (2,0) circle (0.08) node[above] {$1$};

    \draw[fill=red] (0,-1) circle (0.08) node[below] {$1$};
  \end{tikzpicture}\quad
and
\quad
  \begin{tikzpicture}
    \draw[very thick,blue] (-2,0) -- (-1,0);
    \draw[very thick,blue] (-1,0) -- (0,0);
    \draw[very thick,blue] (0,0) -- (1,0);
    \draw[very thick,blue] (1,0) -- (2,0);
    \draw[very thick,blue] (0,0) -- (0,-1);

    \draw[fill=red] (-2,0) circle (0.08) node[above] {$1$};
    \draw[fill=red] (-1,0) circle (0.08) node[above] {$2$};
    \draw[fill=red] (0,0) circle (0.08) node[above] {$3$};
    \draw[fill=red] (1,0) circle (0.08) node[above] {$2$};
    \draw[fill=red] (2,0) circle (0.08) node[above] {$1$};

    \draw[fill=red] (0,-1) circle (0.08) node[below] {$2$};
  \end{tikzpicture}
\end{center}

  where the numbers above a vertex indicate its multiplicity in
  $Z_{\tmop{num}}$. We have thus shown that for any graph $\Gamma$, $b_i > 2$
  implies that the multiplicity of $E_i$ in $Z_{\tmop{num}}$ is $1$.
\end{proof}

We denote by $\alpha = ( \alpha_1, \ldots, \alpha_r )$ the vector whose
entries are the discrepancies $\alpha_i$ of the exceptional curves of $\pi_0 :
Z_0 \longrightarrow Z$. The vector $\alpha$ can be expressed in terms of the
intersection matrix $I = ( E_i E_j )$ and the vector $v_I = ( E_1^2 + 2,
\ldots, E_r^2 + 2 )$ as
\begin{eqnarray}
  \alpha & = & - I^{- 1} v_I .  \label{discrepancyI}
\end{eqnarray}
Note that these $\alpha_i$'s are different from the discrepancies of the order
$a_i$ introduced earlier.

\begin{proposition}
  \label{estimateonD2}Let $D = \sum_{i = 1}^r n_i E_i$ be an effective
  exceptional divisor on $Z_0$ and $s$ be the minimal integer such that $D
  \leqslant s Z_{\tmop{num}}$. Then
  \begin{eqnarray*}
    - D^2 & \geqslant & 2 s + \sum_{i = 1}^r ( b_i - 2 ) n_i .
  \end{eqnarray*}
\end{proposition}

\begin{proof}
  The adjunction formula for a divisor on a surface gives
  \begin{eqnarray*}
    - D^2 & = & 2 h^0 ( \OO_D ) + K_{\tilde{Z}} D,
  \end{eqnarray*}
  and from equation (\ref{discrepancyI}), we compute $K_{\tilde{Z}} D =
  \sum_{i = 1}^r ( b_i - 2 ) n_i$. We show by induction that if $s$ is the
  minimal integer such that $D \leqslant s Z_{\tmop{num}}$ then $h^0 ( \OO_D )
  \geqslant s$, which completes the proof. The implication is trivial for $s =
  0$, and we suppose that it holds for $s - 1$. By Lemma \ref{firstlemma}, we
  have $h^0 ( \OO_D ) \geqslant h^0 ( \OO_{D'} )$ where $D'$ is the numerical
  cycle associated to $D$. Now $- D'$ is $\sigma$-nef, so $D' -
  Z_{\tmop{num}}$ is effective. Taking Euler characteristics of the exact
  sequence, $0 \longrightarrow \OO_{D' - Z_{\tmop{num}}} ( - Z_{\tmop{num}} )
  \longrightarrow \OO_{D'} \longrightarrow \OO_{Z_{\tmop{num}}}
  \longrightarrow 0$, we obtain $h^0 ( \OO_{D'} ) = \chi ( \OO_{D' -
  Z_{\tmop{num}}} ( - Z_{\tmop{num}} ) ) + h^0 ( \OO_{Z_{\tmop{num}}} )$.
  Since $- Z_{\tmop{num}}$ is $\sigma$-nef and $h^0 ( \OO_{Z_{\tmop{num}}} )
  \geqslant 1$ we have
  \begin{eqnarray*}
    h^0 ( \OO_{D'} ) & \geqslant & h^0 ( \OO_{D' - Z_{\tmop{num}}} ) + 1.
  \end{eqnarray*}
  Now $s$ is the minimal integer such that $D' \leqslant s Z_{\tmop{num}}$, so
  $s - 1$ is the minimal integer such that $D' - Z_{\tmop{num}} \leqslant ( s
  - 1 ) Z_{\tmop{num}}$. By the induction hypothesis, we conclude that $h^0 (
  \OO_D ) \geqslant h^0 ( \OO_{D'} ) \geqslant s - 1 + 1 = s$.
\end{proof}

With this result, we can show that log terminal orders whose centres have
canonical singularities are numerically rational.

\begin{proposition}
  \label{fboundforcanonical}Let $\pi_0 : Z_0 \longrightarrow Z$ be a minimal
  resolution of a canonical surface singularity and $g : \tmmathbf{E}
  \longrightarrow \mathbbm{Q}$ be a function satisfying the hypotheses of
  Theorem \ref{technicaltheorem}. Then $g ( E ) > 0$ for all $E \in
  \tmmathbf{E}^+_0 \backslash \{ 0 \}$ if and only if $g ( Z_{\tmop{num}} ) >
  0$.
\end{proposition}

\begin{proof}
  Let $E$ be an effective divisor on $Z_0$ and $s$ be the minimal integer such
  that $E \leqslant s Z_{\tmop{num}}$. The above proposition applied to the
  case of a canonical surface singularity yields $- E^2 \geqslant 2 s$. For
  the linear term $\ell$ in $g ( E ) = - E^2 + \ell ( E )$, recall that $\ell
  ( E_i ) \leqslant 0$ for all $i$, hence $\ell ( E ) \geqslant \ell ( s
  Z_{\tmop{num}} )$. It follows then
  \begin{eqnarray*}
    g ( E ) & \geqslant & 2 s + \ell ( s Z_{\tmop{num}} ) = s ( -
    Z_{\tmop{num}}^2 + \ell ( Z_{\tmop{num}} ) ) = s g ( Z_{\tmop{num}} ) > 0.
  \end{eqnarray*}
\end{proof}

The final task is to generalise Proposition \ref{fboundforcanonical} for log
terminal singularities with higher multiplicities. Our strategy is to
decompose $D$ as the sum of two effective divisors $D = D_1 + D_2$ with $D_1
D_2 \leqslant 0$. Then we have $g ( D ) \geqslant g ( D_1 ) + g ( D_2 ) - 2
D_1 D_2 \geqslant g ( D_1 ) + g ( D_2 )$ and the problem is reduced to showing
$g ( D_i ) \geqslant 0$ for $i = 1, 2$.

Let $D = \sum_{j \in I} n_j E_j$ with $n_j > 0$ where $I \subseteq [ 1, r ]$,
and we assume $\tmop{supp} D$ is connected. Then $\tmop{supp} D$ contracts to
a log terminal singularity, and we define the multiplicity $m ( D )$ of $D$ to
be the multiplicity of the contracted singularity. The number $m ( D )$ can be
computed by modifying equation (\ref{multiplicityformula}) appropriately,
\begin{eqnarray}
  m ( D ) & = & 2 + \sum_{j \in I} ( b_j - 2 ),  \label{multiplicityeq2}
\end{eqnarray}
and by definition of $D_{\tmop{num}}$ we have $m ( D ) = - D^2_{\tmop{num}}$.
Since $\pi_0$ is a minimal resolution, we have $m ( D ) \geqslant 2$. If $m (
D ) = 2$, then let $D_1 = D$ and $D_2 = 0$. If $m ( D ) > 2$, then let $n$ be
the positive integer $n = \min \{ n_i \mid - E_i^2 > 2 \}$ and define $D_1 =
\gcd ( n D_{\tmop{num}}, D )$, $D_2 = D - D_1$.

\begin{proposition}
  \label{decompositionprop}Let $D$ be a connected effective exceptional
  divisor on $Z_0$ and $D = D_1 + D_2$ be the decomposition above. Let $g :
  \tmmathbf{E} \longrightarrow \mathbbm{Q}$ be a function satisfying the
  hypotheses of Theorem \ref{technicaltheorem}. Then
  \begin{enumerate}
    \item $D_1 D_2 \leqslant 0$
    
    \item If $g ( D_{\tmop{num}} ) > 0$, then $g ( D_1 ) > 0$
    
    \item for each connected component $C$ of $D_2$, we have $m ( C ) < m ( D
    )$.
  \end{enumerate}
\end{proposition}

\begin{proof}
  The proposition is trivial if $m ( D ) = 2$, so we assume $m ( D ) > 2$.
  First note that the effective divisors $n D_{\tmop{num}} - D_1$ and $D_2$
  have no common components, hence $n D_{\tmop{num}} D_2 \geqslant D_1 D_2$.
  Since $D_2$ is supported on $\bigcup_{i \in I} E_i$ and $- D_{\tmop{num}}
  E_i \geqslant 0$ for any $i \in I$, we have $D_{\tmop{num}} D_2 \leqslant
  0$. This proves part 1 of the proposition.
  
  Note that $D_1 \leqslant n D_{\tmop{num}}$, and we now show that $n$ is the
  minimal integer with this property. By definition of $n$, there exists an
  irreducible component $E_s$ of $D$ of multiplicity $n$ and $- E_s^2 > 2$. By
  Proposition \ref{numericalcycleoflogterminalsingularities}, the multiplicity
  of $E_s$ in $D_{\tmop{num}}$ is $1$. Hence the multiplicities of $E_s$ in $n
  D_{\tmop{num}}$ and $D_1$ are equal, so we can conclude that $n$ is the
  minimal integer such that $D_1 \leqslant n D_{\tmop{num}}$. Moreover, the
  multiplicity of $E_j$ in $D_1$ is equal to $n$ whenever $- E_j^2 > 2$. Now
  we can apply Proposition \ref{estimateonD2} and obtain the inequality
  \begin{eqnarray*}
    - D_1^2 & \geqslant & 2 n + \sum_{j \in I} ( b_j - 2 ) n
  \end{eqnarray*}
  and by (\ref{multiplicityformula}) the last expression is equal to $n m (
  D_1 )$. Since $D$ is connected, the same is true for $D_1$, so $m ( D_1 ) =
  - D_{\tmop{num}}^2$. Then by the same argument as in the proof of
  Proposition \ref{fboundforcanonical}, we have
  \begin{eqnarray*}
    g ( D_1 ) & \geqslant & n m ( D_1 ) + \ell ( n D_{\tmop{num}} ) \geqslant
    n ( - D_{\tmop{num}}^2 + \ell ( D_{\tmop{num}} ) ) = n g ( D_{\tmop{num}}
    ) > 0.
  \end{eqnarray*}
  This proves part 2 of the proposition.
  
  Since the multiplicities of $E_s$ in $D_1$ and $D$ are equal, the effective
  divisor $D_2$ is supported away from $E_s$. In particular, any connected
  component $C$ of $D_2$ is supported away from $E_s$. Since $- E_s^2 > 2$, we
  see from (\ref{multiplicityeq2}) that $m ( C )$ must be strictly less than
  $m ( D )$. This proves part 3 of the proposition.
\end{proof}

\begin{note}
  Recall that the divisor $D_{\tmop{num}}$ is a special divisor. 
\end{note}

\subsection{Proof of theorem \ref{technicaltheorem}}

We assume as in the hypothesis of Theorem \ref{technicaltheorem} that $g$ is
positive on special divisors. We wish to show that for all $D \in
\tmmathbf{E}^+_0 \backslash \{ 0 \}$, $g ( D ) > 0$. Clearly we can assume $D$
is connected. Suppose that $m ( D ) = 2$, then since $g ( D_{\tmop{num}} ) >
0$ by hypothesis, we can conclude from Proposition \ref{fboundforcanonical}
that $g ( D ) > 0$. Otherwise, let $D = D_1 + D_2$ be the decomposition from
Proposition \ref{decompositionprop}, we can conclude from parts 1 and 2 of the
same proposition that $g ( D ) \geqslant g ( D_1 ) + g ( D_2 )$, and $g ( D_1
) > 0$. It remains to show that $g ( D_2 ) > 0$. To this end, we repeat the
above argument on each connected component of $D_2$. This procedure terminates
since by part 3 of Proposition \ref{decompositionprop}, the connected
components of $D_2$ have multiplicities strictly less than that of $D$.

\section{\label{section-adjunctionformulafororders}Adjunction formula for
orders}

The adjunction formula for a divisor $D$ on a smooth surface $Z$ expresses the
Euler characteristic of $\OO_D$ in terms of intersection numbers involving $D$
and $K_Z$, (c.f. 4.11, {\cite{reid-chapters}})
\begin{eqnarray}
  \chi ( \OO_D ) & = & - \frac{1}{2} ( K_Z + D ) D.  \label{usualadjunctioneq}
\end{eqnarray}
Note that for the intersection product above to be well-defined, we require
$\tmop{supp} D$ to be a projective variety, and we keep this assumption below.
The aim of this section is to derive a similar adjunction formula for a
terminal order $A$ on $Z$, which expresses the Euler characteristic of $A$
restricted to some divisor $D$ in terms of intersection numbers involving $D$
and $K_A$. As mentioned in the introduction, the following result appears in
the unpublished work of M. Artin and A. J. de Jong.

\begin{theorem}
  \label{orderadjunction}Let $A$ be a terminal order on $Z$ of rank  $r^2$ and
  $D$ be an effective divisor whose support is projective. Then
  \begin{eqnarray}
    \chi ( A \otimes_Z \OO_D ) & = & - \frac{r^2}{2} ( K_A + D ) D 
    \label{adjunctionformulaeq}
  \end{eqnarray}
  where $K_A = K_Z + \Delta_A$ is the canonical divisor of $A$.
\end{theorem}

Note that (\ref{usualadjunctioneq}) appears as a special case of
(\ref{adjunctionformulaeq}) (where $A = \OO_Z$), so we feel justified in
calling (\ref{adjunctionformulaeq}) an adjunction formula. As we have already
seen, our motivation for understanding $\chi ( A \otimes_Z \OO_D )$ is to
study the notion of numerical rationality. In that context, the divisor $D$ is
exceptional with respect to some birational morphism, hence its support is
projective. The rest of this section will be devoted to the proof of Theorem
\ref{orderadjunction}.

\subsection{\label{setupforadjform}Setup}

Let $A$ be a terminal order of rank $r^2$ on a surface $Z$ and $C$ be an
irreducible curve in $Z$. Recall that $Z ( A_C / J ( A_C ) )$ is a product of
field extensions of $k ( C )$ which defines a union of cyclic covers of curves
$\pi_C : \tilde{C} \longrightarrow C$. The degree of $\pi_C$ is of course the
ramification index $e_C$ of $A$ at $C$. Terminal orders can be characterised
using ramification data; an order $A$ is terminal if the ramification divisor
$D = \bigcup D_i$ is a normal crossing divisor on a smooth surface $Z$ and the
cyclic covers $\pi_{D_i}$, $\pi_{D_j}$ ramify only at nodes $p \in D_i \cap
D_j$ with $e_i |e_j$ and $\pi_{D_i}$ totally ramified at $p$.

We first compute $\chi ( A \otimes \OO_C )$ by filtering the sheaf $A$ as
follows. Let $J_C$ be the Jacobson radical of $A \otimes k ( C )$ and $J$ be
its inverse image in $A$. Then $J^e = A ( - C )$ where $e = e_C$ and we have a
filtration $J^e = A ( - C ) \subset J^{e - 1} \subset \cdots \subset J \subset
A$, from which we obtain the exact sequences
\[ 0 \longrightarrow J^{i - 1} / J^i \longrightarrow A / J^i \longrightarrow A
   / J^{i - 1} \longrightarrow 0 \]
for $i = 1, \ldots, e$. Hence
\begin{eqnarray}
  \chi ( A \otimes \OO_C ) & = & \chi ( J^{e - 1} / J^e ) + \chi ( J^{e - 2} /
  J^{e - 1} ) + \cdots + \chi ( J / J^2 ) + \chi ( A / J ) . 
  \label{firstformula}
\end{eqnarray}
We first determine $\chi ( A / J )$, and to do this, we need to know the local
structure of $A / J$.

\subsection{Local structure of $A / J$}

We can use the {\e}tale local structure of $A$ (c.f. {\cite{CI1}}, Definition
2.6) to determine the {\e}tale local structure of $A / J$. Let $r^2$ denote
the rank of $A$ as an $\OO_Z$-module and we assume that $A$ is ramified at
$C$. Since $A$ is terminal, any other ramification curve $D$ intersect $C$
transversely at a finite number of points and the ramification indices satisfy
$e_D |e$ or $e|e_D$.
\begin{enumerate}
  \item First suppose $p \in C$ is a nonsingular point of the ramification
  divisor. Let $u \in \mathfrak{m}_C$ be a uniformising parameter for $\OO_{Z,
  C}$ and we denote $\OO = \OO^{s h}_p$. Then
  \begin{eqnarray*}
    A_p^{s h} & = & M^{r / e \times r / e} \text{ where } M =
    \left(\begin{array}{cccc}
      \OO & \OO & \ldots & \OO\\
      u \OO & \ddots & \ddots & \vdots\\
      \vdots & \ddots & \ddots & \OO\\
      u \OO & \ldots & u \OO & \OO
    \end{array}\right) \subset \OO^{e \times e}\\
    J_p^{s h} & = & M^{r / e \times r / e} \text{ where } N =
    \left(\begin{array}{cccc}
      u \OO & \OO & \ldots & \OO\\
      u \OO & \ddots & \ddots & \vdots\\
      \vdots & \ddots & \ddots & \OO\\
      u \OO & \ldots & u \OO & u \OO
    \end{array}\right) \subset \OO^{e \times e},
  \end{eqnarray*}
  so
  \begin{eqnarray*}
    ( A / J )_p^{s h} & = & \left( ( \OO / u \OO )^{r / e \times r / e}
    \right)^e .
  \end{eqnarray*}
  Moreover $J_p^{s h}$ is generated, as a left (or right) $A^{s h}_p$-module
  by the regular normal element
  \[ \left(\begin{array}{ccccc}
       0 & 1 & 0 & \cdots & 0\\
       0 & 0 & \ddots & \ddots & \vdots\\
       \vdots & \ddots & \ddots & \ddots & 0\\
       0 & \ddots & \ddots & 0 & 1\\
       u & 0 & \cdots & 0 & 0
     \end{array}\right) . \]
  \item Now suppose $p \in C \cap D$ where $D$ is a ramification curve with
  ramification index $e_D$. Let $v \in \mathfrak{m}_D$ be a uniformising
  parameter for $\OO_{Z, D}$. Denote by $S = \OO_{Z, p}^{s h} \left\langle x,
  y \right\rangle / ( x^e - u, y^e - v, x y - \zeta_e y x )$ where $\zeta_e$
  is a primitive $e$-th root of unity. If $e|e_D$, then
  \begin{eqnarray*}
    A_p^{s h} & = & M^{r / e_D \times r / e_D} \text{ where } M =
    \left(\begin{array}{cccc}
      S & S & \ldots & S\\
      y S & \ddots & \ddots & \vdots\\
      \vdots & \ddots & \ddots & S\\
      y S & \ldots & y S & S
    \end{array}\right) \subset S^{e_D / e \times e_D / e}\\
    J_p^{s h} & = & N^{r / e_D \times r / e_D} \text{ where } N =
    \left(\begin{array}{cccc}
      x S & x S & \ldots & x S\\
      x y S & \ddots & \ddots & \vdots\\
      \vdots & \ddots & \ddots & x S\\
      x y S & \ldots & x y S & x S
    \end{array}\right) \subset S^{e_D / e \times e_D / e}
  \end{eqnarray*}
  Let $\overline{S} = S / x S \simeq k \{ v \} [ y ] / ( y^e - v )$, where $k
  \{ v \}$ denotes the strict henselisation of $k [ v ]$ at the origin. Then 
  \begin{eqnarray*}
    ( A / J )_{\mathfrak{p}}^{s h} & = & P^{r / e_D \times r / e_D} \text{
    where } P = \left(\begin{array}{cccc}
      \bar{S} & \bar{S} & \ldots & \bar{S}\\
      y \bar{S} & \ddots & \ddots & \vdots\\
      \vdots & \ddots & \ddots & \bar{S}\\
      y \bar{S} & \ldots & y \bar{S} & \bar{S}
    \end{array}\right) \subset \overline{S}^{e_D / e \times e_D / e} .
  \end{eqnarray*}
  The generator for $J_p^{s h}$ in this case is just $x 1_{A_p^{s h}}$.
  
  \item In the case where $e_D |e$, we denote by $S = \OO_{Z, p}^{s h}
  \left\langle x, y \right\rangle / ( x^{e_D} - u, y^{e_D} - v, x y -
  \zeta_{e_D} y x )$. Then

  \begin{eqnarray*}
    A_p^{s h} & = & M^{r / e \times r / e} \text{ where } M =
    \left(\begin{array}{cccc}
      S & S & \ldots & S\\
      x S & \ddots & \ddots & \vdots\\
      \vdots & \ddots & \ddots & S\\
      x S & \ldots & x S & S
    \end{array}\right) \subset S^{e / e_D \times e / e_D}\\
    J_p^{s h} & = & N^{r / e \times r / e} \text{ where } N =
    \left(\begin{array}{cccc}
      x S & S & \ldots & S\\
      x S & \ddots & \ddots & \vdots\\
      \vdots & \ddots & \ddots & S\\
      x S & \ldots & x S & x S
    \end{array}\right) \subset S^{e / e_D \times e / e_D}
  \end{eqnarray*}
  so
  \begin{eqnarray*}
    ( A / J )_p^{s h} & = & ( ( S / x S )^{e / e_D} )^{r / e \times r / e} .
  \end{eqnarray*}
  Again $J^{s h}_p$ is generated by a regular normal element
  \[ \left(\begin{array}{ccccc}
       0 & 1 & 0 & \cdots & 0\\
       0 & 0 & \ddots & \ddots & \vdots\\
       \vdots & \ddots & \ddots & \ddots & 0\\
       0 & \ddots & \ddots & 0 & 1\\
       x & 0 & \cdots & 0 & 0
     \end{array}\right) . \]
\end{enumerate}
Note that in each case above, the ideal $J$ is generated locally by a regular
normal element.

\begin{lemma}
  Let $\pi : \tilde{C} \longrightarrow C$ be the cover of $C$ determined by
  the ramification data. Then
  \begin{eqnarray*}
    \chi ( \OO_{\tilde{C}} ) & = & e \chi ( \mathcal{O}_C ) - \frac{e}{2}
    \sum_{D \in Z^1 \backslash \{ C \}} \left( 1 - \frac{1}{\min \{ e, e_D \}}
    \right) .
  \end{eqnarray*}
\end{lemma}

\begin{proof}
  The cover $\pi$ has degree $e = e_C$, so by the Riemann-Hurwitz formula, we
  have
  \begin{eqnarray*}
    \chi ( \OO_{\tilde{C}} ) & = & e \chi ( \OO_C ) - \frac{1}{2} \sum_{p \in
    \tilde{C}} ( e_p - 1 ) .
  \end{eqnarray*}
  If $\pi ( p )$ is a nonsingular point of the ramification divisor, then $e_p
  = 1$. Now suppose $\pi ( p ) \in C \cap D$ where $A$ is ramified on $D$ with
  ramification index $e_D$. If $e_D \geqslant e$, then $\pi$ is totally
  ramified at $\pi ( p )$, hence $e_p = e$. If $1 < e_D < e$, then in the
  fibre $\pi^{- 1} ( \pi ( p ) )$ there are $e / e_D$ points each with
  ramification index $e_D - 1$. A simple calculation then yields the above
  formula.
\end{proof}

\begin{lemma}
  The sheaf $A / J$ considered as a sheaf on $C$ is a $\pi_{\ast}
  \OO_{\tilde{C}}$-module.
\end{lemma}

\begin{proof}
  Firstly, $A / J \otimes k ( C )$ is isomorphic to $M_n ( k ( \tilde{C} ) )$.
  So it suffices to show that $( A / J )_{\mathfrak{p}}
  \mathcal{O}_{\tilde{C}, \mathfrak{p}} \subseteq ( A / J )_{\mathfrak{p}}$
  for all prime ideals $\mathfrak{\mathfrak{p}} \in \tmop{Spec} C$ where we
  identify everything with their natural images in $A / J \otimes k ( C )$.
  From the {\e}tale local structures for $A / J$ above, we can see that $( A /
  J )_{\mathfrak{p}}^{s h} \mathcal{O}_{\tilde{C}, \mathfrak{p}}^{s h}
  \subseteq ( A / J )_{\mathfrak{p}}^{s h}$ for all $\mathfrak{p} \in
  \tmop{Spec} C$. Intersecting with $A / J \otimes k ( C )$ gives the desired
  result.
\end{proof}

\begin{proposition}
  {\tmdummy}
  
  \begin{eqnarray*}
    \chi ( A / J ) & = & \frac{r^2}{2 e} \left( 2 \chi ( \OO_C ) - C \cdot
    \Delta_A + \left( 1 - \frac{1}{e} \right) C^2 \right)
  \end{eqnarray*}
\end{proposition}

\begin{proof}
  The previous lemma shows that we can consider $A / J$ as a sheaf on
  $\tilde{C}$. In fact, we can see from the local structure of $A / J$ that it
  is an order on $\tilde{C}$ in the semi-simple algebra $M_{r / e} ( k (
  \tilde{C} ) )$ (semi-simple since $k ( \tilde{C} )$ is a product of fields).
  Hence we can embed $A / J$ in a maximal order $\Omega$. This gives an exact
  sequence of $\OO_C$-modules
  \[ 0 \longrightarrow A / J \longrightarrow \pi_{\ast} \Omega \longrightarrow
     Q \longrightarrow 0 \]
  where $Q$ is a torsion sheaf supported on points where $A / J$ is not a
  maximal order on the corresponding fibre. Since the Brauer group of a curve
  is trivial, $\Omega$ is a maximal order in a matrix algebra, hence is
  trivial Azumaya. This gives $\chi ( \Omega ) = ( r / e )^2 \chi (
  \OO_{\tilde{C}} )$ which is equal to $\chi ( \pi_{\ast} \Omega )$ since
  $\pi$ is a finite morphism. The sheaf $Q$ is supported on points, so $\chi (
  Q )$ is the sum of the lengths of $Q_p$ over $p \in C$. Referring again to
  the local structure of $A / J$, we see that $A / J$ is nonmaximal at $p$ if
  and only if $p$ is a point of intersection of $C$ and a ramification curve
  $D$ where $e|e_D$. A simple computation shows that
  \begin{eqnarray*}
    Q_p & = & P^{r / e_D \times r / e_D} \text{ where } P =
    \left(\begin{array}{cccc}
      0 & 0 & \cdots & 0\\
      k & \ddots & \ddots & \vdots\\
      \vdots & \ddots & \ddots & \vdots\\
      k & \cdots & k & 0
    \end{array}\right) \subset k^{e_D / e \times e_D / e}
  \end{eqnarray*}
  hence
  \begin{eqnarray*}
    \dim ( Q_p ) & = & \frac{r^2}{e_D^2} \sum_{j = 1}^{e_D / e - 1} j =
    \frac{r^2}{2 e} \left( \frac{1}{e} - \frac{1}{e_D} \right) .
  \end{eqnarray*}
  So
  \begin{eqnarray*}
    \chi ( Q ) & = & \sum_{D \in Z^1 \backslash \{ C \}} \frac{r^2}{2 e}
    \left( \frac{1}{e} - \frac{1}{\max \{ e, e_D \}} \right) C \cdot D.
  \end{eqnarray*}
  Combining the expressions for $\chi ( \Omega )$ and $\chi ( Q )$, we obtain
  \begin{eqnarray*}
    \chi ( A / J ) & = & \frac{r^2}{e} \left( \chi ( \mathcal{O}_C ) -
    \frac{1}{2} \sum_{D \in Z^1 \backslash \{ C \}} \left( 1 - \frac{1}{\min
    \{ e, e_D \}} + \frac{1}{e} - \frac{1}{\max \{ e, e_D \}} \right) C \cdot
    D \right)\\
    & = & \frac{r^2}{e} \left( \chi ( \mathcal{O}_C ) - \frac{1}{2} \sum_{D
    \in Z^1 \backslash \{ C \}} \left( 1 - \frac{1}{e_D} \right) C \cdot D
    \right)
  \end{eqnarray*}
  which proves the proposition.
\end{proof}

To finish the computation of $\chi ( A \otimes_Z \OO_C )$ we need the values
of $\chi ( J^i / J^{i + 1} )$. Let $B$ be an $\mathcal{O}_C$-algebra. We say
that $L$ is an invertible $( B - B )$-bimodule if there exists a $( B - B
)$-bimodule $L'$ such that $L \otimes_B L' \simeq B$ as $( B - B )$-bimodules.

\begin{lemma}
  \label{bimodulechiformula}Let $C$ be a projective curve, $B$ be an
  $\OO_C$-algebra which is torsion-free as an $\OO_C$-module, and $L, L'$ be
  invertible $( B - B )$-bimodules. Then
  \begin{eqnarray*}
    \chi ( L \otimes_B L' ) & = & \chi ( L ) + \chi ( L' ) - \chi ( B ) .
  \end{eqnarray*}
\end{lemma}

\begin{proof}
  First we assume that $L$ is generated as an $\OO_C$-module by its sections.
  Suppose $p \in C$ is a closed point and $s_p \in L_p$ is a regular element,
  that is $B_p \longrightarrow L_p$ given by $b \longmapsto s_p b$ is
  injective. Then since $L$ is generated by sections, we can lift this to a
  section $s \in H^0 ( C, L )$. Now $B$ is torsion-free as an
  $\mathcal{O}_C$-module, so the map $B \longrightarrow L$ given by $s$ is
  also injective. This gives an exact sequence $0 \longrightarrow B
  \longrightarrow L \longrightarrow Q \longrightarrow 0$ of right $B$-modules.
  Since $- \otimes_B L'$ induces an equivalence of categories, it is exact, so
  we have the exact sequence
  \[ 0 \longrightarrow L' \longrightarrow L \otimes_B L' \longrightarrow Q
     \otimes_B L' \longrightarrow 0. \]
  Note that since $L'$ is invertible, its rank as an $\OO_C$-module is the
  same as the $\mathcal{O}_C$-rank of $B$. Thus the sheaf $Q$ is supported on
  points, and $\chi ( Q \otimes_B L' ) = \chi ( Q )$. This gives
  \begin{eqnarray*}
    \chi ( L \otimes_B L' ) & = & \chi ( L' ) + \chi ( Q ) = \chi ( L' ) +
    \chi ( L ) - \chi ( B ) .
  \end{eqnarray*}
  In general, let $\OO_C ( 1 )$ be a very ample line bundle on $C$ so that
  $\OO_C ( n ) \otimes_C L$ is generated by sections for some $n \gg 0$. We
  denote by $r$ the rank of $L$ as an $\OO_C$-module, and note that $L
  \otimes_B L'$ has the same rank. Then
  \begin{eqnarray*}
    \chi ( \OO_C ( n ) \otimes_C L \otimes_B L' ) & = & \chi ( \OO_C ( n )
    \otimes_C L ) + \chi ( L' ) - \chi ( B )\\
    & = & r n + \chi ( L ) + \chi ( L' ) - \chi ( B ) .
  \end{eqnarray*}
  But $\chi ( \OO_C ( n ) \otimes_C L \otimes_B L' ) = r n + \chi ( L
  \otimes_B L' )$ so we are done. 
\end{proof}

Recall that $J$ is generated locally by a regular normal element, hence so is
$J^i / J^{i + 1}$. The following lemma follows from the local structure of $A
/ J$.

\begin{lemma}
  \label{Jsheaves}The sheaves $J^i / J^{i + 1}$ are invertible $( A / J - A /
  J )$-bimodules, locally generated by a regular normal element. Moreover $J^i
  / J^{i + 1} \otimes_{A / J} J^k / J^{k + 1} \simeq J^{i + k} / J^{i + k +
  1}$ as $( A / J - A / J )$-bimodules.
\end{lemma}

\begin{corollary}
  \label{chiformulaforArestrictedtoC} {\tmdummy}
  
  \begin{eqnarray*}
    \chi ( A \otimes_Z \OO_C ) & = & \frac{r^2}{2} \left( 2 \chi ( \OO_C ) - C
    \cdot \Delta_A \right)
  \end{eqnarray*}
\end{corollary}

\begin{proof}
  By Lemma \ref{Jsheaves}, we have $( J / J^2 )^{\otimes e} \simeq J^e / J^{e
  + 1} \simeq A / J \otimes_Z \OO_Z ( - C )$. By Lemma
  \ref{bimodulechiformula}, we have $\chi ( A / J \otimes_Z \OO_Z ( - C ) ) =
  \chi ( ( J / J^2 )^{\otimes e} ) = e \chi ( J / J^2 ) - ( e - 1 ) \chi ( A /
  J )$. Note that the rank of $A / J$ as an $\OO_C$-module is $r^2 / e$, so
  $\chi ( A / J \otimes_Z \OO_Z ( - C ) ) = \chi ( A / J ) - r^2 C^2 / e$.
  Putting these together, we get
  \begin{eqnarray*}
    \chi ( J / J^2 ) & = & - \frac{r^2 C^2}{e^2} + \chi ( A / J ) .
  \end{eqnarray*}
  Recall (\ref{firstformula}) from the beginning of this section,
  \begin{eqnarray*}
    \chi ( A \otimes_Z \OO_C ) & = & \chi ( J^{e - 1} / J^e ) + \chi ( J^{e -
    2} / J^{e - 1} ) + \cdots + \chi ( J / J^2 ) + \chi ( A / J )\\
    & = & \chi ( A / J ) + \sum_{i = 1}^{e - 1} \left( i \chi ( J / J^2 ) - (
    i - 1 ) \chi ( A / J ) \right)\\
    & = & \chi ( A / J ) + \frac{e ( e - 1 )}{2}  \left( - \frac{r^2
    C^2}{e^2} + \chi ( A / J ) \right) - \frac{( e - 1 ) ( e - 2 )}{2} \chi (
    A / J )\\
    & = & - \frac{e - 1}{2 e} r^2 C^2 + \left( 1 + \frac{e ( e - 1 )}{2} -
    \frac{( e - 1 ) ( e - 2 )}{2} \right) \chi ( A / J )\\
    & = & - \frac{e - 1}{2 e} r^2 C^2 + e \chi ( A / J )\\
    & = & \frac{r^2}{2} \left( - ( 1 - \frac{1}{e} ) C^2 + \left( 2 \chi (
    \OO_C ) - C \cdot \Delta_A + \left( 1 - \frac{1}{e} \right) C^2 \right)
    \right)\\
    & = & \frac{r^2}{2} \left( 2 \chi ( \OO_C ) - C \cdot \Delta_A \right)
  \end{eqnarray*}
\end{proof}

\subsection{Proof of Theorem \ref{orderadjunction}}

It remains to prove Theorem \ref{orderadjunction} in the case of a general
effective divisor $E$.

\begin{lemma}
  \label{chiformulaforE}Let $E = n_1 E_1 + \ldots + n_s E_s$ be an effective
  divisor on $Z$ and $V$ be a rank $r$ vector bundle on $Z$. Then
  \begin{eqnarray}
    \chi ( V \otimes_Z \OO_E ) & = & - \frac{r E^2}{2} + \sum_{i = 1}^s n_i
    \left( \frac{r E_i^2}{2} + \chi ( V \otimes_Z \OO_{E_i} ) \right) . 
    \label{chiformulaforEeq}
  \end{eqnarray}
\end{lemma}

\begin{proof}
  We prove (\ref{chiformulaforEeq}) by induction. For irreducible $E$, there
  is nothing to prove. Suppose (\ref{chiformulaforEeq}) holds for $E$, we show
  that it holds too for $E + E_j$. Using the following exact sequence
  \[ 0 \longrightarrow V \otimes_Z \OO_{E_j} ( - E ) \longrightarrow V
     \otimes_Z \OO_{E + E_j} \longrightarrow V \otimes_Z \OO_E \longrightarrow
     0 \]
  we obtain
  \begin{eqnarray*}
    \chi ( V \otimes_Z \OO_{E + E_j} ) & = & \chi ( V \otimes_Z \OO_{E_j} ( -
    E ) ) - \frac{r E^2}{2} + \sum_{i = 1}^s n_i \left( \frac{r E_i^2}{2} +
    \chi ( V \otimes_Z \OO_{E_i} ) \right) .
  \end{eqnarray*}
  Since $\chi ( V \otimes_Z \OO_{E_j} ( - E ) ) = \chi ( V \otimes_Z \OO_{E_j}
  ) - r E_j \cdot E$, we get
  \begin{eqnarray*}
    \chi ( V \otimes_Z \OO_{E + E_j} ) & = & \chi ( V \otimes_Z \OO_{E_j} ) +
    \frac{r E_j^2}{2} - \frac{r ( E + E_j )^2}{2} + \sum_{i = 1}^s n_i \left(
    \frac{r E_i^2}{2} + \chi ( V \otimes_Z \OO_{E_i} ) \right)
  \end{eqnarray*}
  which shows that (\ref{chiformulaforEeq}) holds for $E + E_j$. 
\end{proof}

Theorem \ref{orderadjunction} follows then directly from Corollary
\ref{chiformulaforArestrictedtoC} and Lemma \ref{chiformulaforE}.
\begin{eqnarray*}
  \chi ( A \otimes_Z \OO_E ) & = & - \frac{r^2 E^2}{2} + \sum_{i = 1}^s n_i
  \left( \frac{r^2 E_i^2}{2} + \frac{r^2}{2} \left( 2 \chi ( \OO_{E_i} ) - E_i
  \cdot \Delta_A \right) \right)\\
  & = & \frac{r^2}{2} \left( - E^2 + \sum_{i = 1}^s n_i \left( E_i^2 - ( K_A
  + E_i ) E_i \right) \right)\\
  & = & - \frac{r^2}{2} ( K_A + E ) E
\end{eqnarray*}

\section{Concluding Remarks}

We have shown that the notion of numerical rationality includes many
interesting examples of orders which arise naturally in the context of
noncommutative birational geometry. Our definition is natural in that it does
not depend on the choice of resolution, nor does it depend on the choice of
representative in a Morita equivalence class (if the centre has rational
singularities). Moreover, the adjunction formula for orders makes it easy to
check whether an order is numerically rational.

\end{document}